\documentclass[11pt]{amsart}
\usepackage{graphicx}
\usepackage{amsmath} 
\usepackage{amsthm,amsfonts,amssymb,mathrsfs,amscd,amstext,amsbsy} 
\usepackage{epic,eepic} 
\usepackage{yfonts}
\usepackage{paralist,enumerate}
\usepackage[all]{xy}
\usepackage{hyperref}
\hypersetup{colorlinks}

\newcommand{\D}{\displaystyle}

\newtheorem{theorem}{Theorem}[section]
\newtheorem{cor}[theorem]{Corollary}

\newtheorem{theo}[theorem]{Theorem}
\newtheorem{lem}[theorem]{Lemma}

\newtheorem{que}[theorem]{Question}

\newtheorem{Definition}[theorem]{Definition}
\newtheorem*{Definition*}{Definition}

\def\qed{\hfill \ifhmode\unskip\nobreak\fi\quad\ifmmode\Box\else$\Box$\fi\\ }

\begin{document}

\title[Symplectic circle actions with isolated fixed points]{Symplectic circle actions with isolated fixed points}
\author{Donghoon Jang}
\address{Department of Mathematics, University of Illinois at Urbana-Champaign, Urbana, IL 61801}
\email{jang12@illinois.edu}

\begin{abstract}
Consider a symplectic circle action on a closed symplectic manifold with non-empty isolated fixed points. Associated to each fixed point, there are well-defined non-zero integers, called \emph{weights}. We prove that the action is Hamiltonian if the sum of an odd number of weights is never equal to zero (the weights may be taken at different fixed points). Moreover, we show that if $\dim M=6$, or if $\dim M=2n \leq 10$ and each fixed point has weights $\{\pm a_1, \cdots, \pm a_n\}$ for some positive integers $a_i$, it is enough to consider the sum of three weights. As applications, we recover the results for semi-free actions from \cite{TW}, and for certain circle actions on six-dimensional manifolds from \cite{G}.







\end{abstract}

\maketitle

\section{Introduction}

$\indent$



One important question in symplectic geometry is when a symplectic group action is Hamiltonian. In particular, let the circle act symplectically on a closed symplectic manifold. In some cases, the existence of a fixed point implies that the action is Hamiltonian. For instance, if the manifold is K\"ahler \cite{F}, if the dimension of the manifold is less than or equal to four \cite{MD}, if there are three fixed points \cite{J}, or if the circle action is semi-free with discrete fixed points \cite{TW}. In \cite{MD}, D. McDuff constructs a non-Hamiltonian symplectic circle action on a six-dimensional symplectic manifold with fixed tori. Therefore, if the dimension of a manifold is greater than or equal to six, we require more restrictions for a symplectic action to be Hamiltonian.


In this paper, we focus on symplectic circle actions on closed symplectic manifolds with isolated fixed points. Let the circle act symplectically on a $2n$-dimensional closed symplectic manifold and suppose that the fixed points are isolated. Associated to each fixed point $p$, there are well-defined non-zero integers $w_p^i$, called \emph{weights}, $1 \leq i \leq n$. We prove that if the weights at the fixed points satisfy certain conditions, then the action is Hamiltonian. Consider a collection of weights among all the fixed points, counted with multiplicity. For each integer $a$, the multiplicity of $a$ in the collection is therefore precisely $\max_{p \in M^{S^1}} |\{i|a=w_p^i\}|$. First, we show that the symplectic action is Hamiltonian if the sum of an odd number of weights in the collection is never equal to zero.





\begin{theo} \label{t42}
Consider a symplectic circle action on a closed symplectic manifold with non-empty isolated fixed points. The action is Hamiltonian if the sum of an odd number of weights among all fixed points is never equal to zero.

\end{theo}


For instance, if the action is semi-free, all the weights are either $+1$ or $-1$. Therefore, the sum of an odd number of weights cannot equal zero and hence the action is Hamiltonian. In some cases, we only need to look at the sum of three weights.

\begin{theo} \label{t44}
Consider a symplectic circle action on a $2n$-dimensional closed symplectic manifold with non-empty fixed points, whose weights are $\{\pm a_1, \pm a_2, \cdots, \pm a_n\}$ for some positive integers $a_i$, for $1 \leq i \leq n$. Assume that $n \leq 5$ and $\pm a_i \pm a_j \pm a_k \neq 0$ for all $i<j<k$. Then the action is Hamiltonian.
\end{theo}

\begin{theo} \label{t43}
Consider a symplectic circle action on a six-dimensional closed symplectic manifold with non-empty isolated fixed points. The action is Hamiltonian if the sum of three weights among all fixed points is never equal to zero.
\end{theo}



It is an intriguing question concerning symplectic circle actions, if there exists a non-Hamiltonian symplectic circle action on a compact symplectic manifold with non-empty discrete fixed points. The condition that the sum of three weights among all fixed points is never equal to zero, seems to play a certain role for a symplectic circle action to be Hamiltonian. If a symplectic circle action on a closed symplectic manifold $M$ has two fixed points, then either $M$ is the 2-sphere, or $\dim M=6$ and the weights at the two fixed points are $\{-a-b,a,b\}$ and $\{-a,-b,a+b\}$ for some positive integers $a$ and $b$ \cite{K}, \cite{PT}. If $\dim M=6$, then the action cannot be Hamiltonian, since a compact Hamiltonian $S^1$-manifold $M$ has at least $\frac{1}{2} \dim M+1$ fixed points. Moreover, there is the sum of three weights that is equal to zero. In fact, the first Chern class at each fixed point, which is the sum of weights at the fixed point, is equal to zero. However, to the author's knowledge, we do not know, whether such a manifold exists or not.


\begin{que} Let the circle act on a closed almost complex manifold with non-empty isolated fixed points. Suppose that the sum of three weights among all fixed points is never equal to zero. Then is the action Hamiltonian? \end{que}

\section{Background and Notation}

A manifold $M$ is called a \textbf{symplectic manifold}, if there is a closed, non-degenerate two-form $\omega$ on $M$, called a \textbf{symplectic form}. Let the circle act on a symplectic manifold $(M,\omega)$. The action is called \textbf{symplectic}, if it preserves the symplectic form $\omega$. Let $X_{M}$ be the vector field on $M$ generated by the circle action. The action is called \textbf{Hamiltonian}, if there is a map $\mu : M \rightarrow \mathbb{R}$ such that $d\mu=\iota_{X_{M}}\omega$.

Let the circle act symplectically on a $2n$-dimensional symplectic manifold $M$. Suppose that $p$ is an isolated fixed point. We can identify $T_p M$ with $\mathbb{C}^n$ and the action of $S^1$ at $p$ with $\lambda \cdot (z_1, \cdots, z_n)=(\lambda^{\xi_{p}^{1}} z_1, \cdots, \lambda^{\xi_{p}^{n}} z_n)$, where $\xi_{p}^{i}$ are non-zero integers. These non-zero integers are called \textbf{weights} at the fixed point $p$. Let $\lambda_{p}$ be twice of the number of negative weights at $p$. This is called the \textbf{index} of $p$. Denote $\sigma_i$ by the elementary symmetric polynomial of degree $i$ in $n$ variables. In \cite{L}, P. Li shows the following:


\begin{theo} \label{t28} \cite{L}
Consider a symplectic circle action on a $2n$-dimensional closed symplectic manifold $M$ with isolated fixed points. Then
\begin{center}
$\D{\chi^i(M) = \sum_{p \in M^{S^{1}}} \frac{\sigma_i (t^{\xi_{p}^{1}}, \cdots, t^{\xi_{p}^{n}})}{\prod_{j=1}^{n} (1-t^{\xi_{p}^{j}})}=(-1)^i N^i=(-1)^{n-i}N^{n-i}}$,
\end{center}
where $N^i$ is the number of fixed points of index $2i$ and $t$ is an indeterminate. Moreover, $\chi^0(M)=1$ if the action is Hamiltonian, and $\chi^0(M)=0$ if it is not Hamiltonian.
\end{theo}

\section{Symplectic Circle Actions with Isolated Fixed Points}

Let the circle act symplectically on a closed symplectic manifold $M$ with isolated fixed points. Denote by $A=\{a_1, a_2, \cdots , a_l\}$ the collection of all the absolute values of weights among all the fixed points counted with multiplicity, and $A_i=\{a_{j_1}+a_{j_2}+\cdots+a_{j_i}\}_{a_{j_1}<a_{j_2}<\cdots<a_{j_i}}$ the collection of sums of $i$ elements of $A$, for $1 \leq i \leq l$. For each positive integer $a$, the multiplicity of $a$ in $A$ is $\max_{p \in M^{S^1}} |\{i|a=|w_p^i|\}|$. Note that here we consider the collection of the absolute values of the weights, and hence it is different from the one in the introduction. We begin with the following result, which implies Theorem \ref{t42}.



\begin{theo} \label{t12}

Let the circle act symplectically on a closed symplectic manifold $M$ with non-empty isolated fixed points. Denote by $A=\{a_1, a_2, \cdots , a_l\}$ the collection of all the absolute values of weights among all the fixed points, counted with multiplicity, and $\displaystyle{A_i=\{a_{j_1}+a_{j_2}+\cdots+a_{j_i}\}_{a_{j_1}<a_{j_2}<\cdots<a_{j_i}}}$ the collection of sums of $i$-elements of $A$, for $1 \leq i \leq l$. If there exists $0< i < n$ such that $A_i \cap A_j = \emptyset$ for all $j$ such that $j \neq i \mod 2$, then the action is Hamiltonian.
\end{theo}

\begin{proof}

Assume, on the contrary, that the action is not Hamiltonian. By Theorem \ref{t28}, $\chi^0(M)=\chi^n(M)=0$ and there are no fixed points of index $0$ and $2n$. Moreover,
\begin{center}
$\D{\chi^0(M) = \sum_{p \in M^{S^{1}}} \frac{1}{\prod_{m=1}^{n} (1-t^{\xi_p^m})}=\sum_{p \in M^{S^{1}}}\frac {\prod_{-\xi_p^m<0} (-t^{-\xi_p^m})}{\prod_m (1-t^{|\xi_p^m|})}}$

$\displaystyle{=\sum_{p \in M^{S^{1}}} (-1)^{\frac{\lambda_{p}}{2}} \frac {\prod_{\xi_p^m<0} t^{-\xi_p^m}}{\prod_m (1-t^{|\xi_p^m|})}=0}$.
\end{center}

Denote by $B_p=A \setminus \{|w_p^1|,\cdots,|w_p^n|\}$ the set of elements in $A$ minus the absolute values of the weights at $p$. We multiply the equation above by $\prod_{i=1}^{l} (1-t^{a_i})$ to get
\begin{center}
$\D{\sum_{p \in M^{S^{1}}} (-1)^{\frac{\lambda_{p}}{2}} \prod_{\xi_p^m<0} t^{-\xi_p^m} \prod_{a \in B_p} (1-t^{a})=0}$.
\end{center}

In the last equation, each term whose exponent is the sum of odd elements has the coefficient $-1$ and each term whose exponent is the sum of even elements has the coefficient $1$. Since $\chi^0(M)=0$, this implies that each term whose exponent is the sum of an odd number of elements must cancel out with another term whoch exponent is the sum of an even number of elements.

Suppose that there is a fixed point $p$ of index $2i$. Then $p$ contributes a summand $(-1)^i t^{-\sum_{\xi_p^m<0} \xi_p^m}$, where $\xi_p^m<0$ are the negative weights at $p$. Since $\chi^0(M)=0$, this term must be cancelled out. The coefficient of the term is $(-1)^i$. Therefore, if the term is cancelled out by another term, then its exponent must be the sum of $j$-elements, where $j$ and $i$ have different parities. By the assumption that $A_i \cap A_j = \emptyset$ for $j \neq i \mod 2$, the summand $(-1)^i t^{-\sum_{\xi_p^m<0} \xi_p^m}$ cannot be cancelled out, which is a contradiction. Therefore, there are no fixed points of index $2i$.


From now on we seperate into several cases, depending on if $i \leq \frac{n}{2}$ or $i > \frac{n}{2}$, if $i$ is odd or even, and if $n$ is odd or even. Each case is a slight variation of the other cases. If $i > \frac{n}{2}$, we use the symmetry that $N^j=N^{n-j}$ for all $j$, where $N_j$ is the number of fixed points of index $2j$. With the symmetry, the case where $i > \frac{n}{2}$ is a slight variation of the case where $i \leq \frac{n}{2}$.




First, suppose that $i \leq \frac{n}{2}$ and $i$ is odd. By Theorem \ref{t28},
\begin{center}
$\D{\chi^i(M) = \sum_{p \in M^{S^{1}}} \frac{\sigma_i (t^{\xi_p^1}, \cdots, t^{\xi_p^n})}{\prod_m (1-t^{\xi_p^m})}}$

$\D{=\sum_{p \in M^{S^{1}}} (-1)^{\frac{\lambda_{p}}{2}} \frac {[\prod_{\xi_p^m<0} t^{-\xi_p^m}] \sigma_i (t^{\xi_p^1}, \cdots, t^{\xi_p^n})}{\prod_m (1-t^{|\xi_p^m|})}=0}$.
\end{center}

We multiply the equation above by $\prod_{i=1}^l (1-t^{a_i})$ to get

\begin{center}
$\D{\sum_{p \in M^{S^{1}}} (-1)^{\frac{\lambda_{p}}{2}} \sigma_i (t^{\xi_p^1}, \cdots, t^{\xi_p^n}) \prod_{\xi_p^m<0} t^{-\xi_p^l}  \prod_{a \in B_p} (1-t^{a})=0}$.
\end{center}

Suppose that a fixed point $p$ has the index $2k$, where $k$ is even and $0 \leq k \leq 2i$. In the last equation, such a point contributes a summand whose exponent is the sum of $i$ elements. By using the argument as above, such a term cannot be cancelled out. Hence there are no fixed points of index $0,4,\cdots,4i$, i.e. $N^0=N^4=\cdots=N^{4i}=0$ and thus $N^{2n}=N^{2n-4}=\cdots=N^{2n-4i}=0$, by Theorem \ref{t28}. In particular, $\chi^{i+1}(M)=(-1)^{i+1}N^{i+1}=0$.

Next, we consider $\chi^{i+1}(M)=0$. Using the same argument, one can show that there are no fixed points of index $2k$ where $k$ is odd and $0 \leq k \leq i+2$. And then we consider $\chi^{i+2}(M)=0$ to conclude that there are no fixed points of index $2k$ where $k$ is even and $0 \leq k \leq i+3$. We continue this to conclude that there are no fixed points of any index, which is a contradiction.

Second, suppose that $i \leq \frac{n}{2}$ and $i$ is even. Using the same argument as in the first case, by considering $\chi^i(M)=0$, one can show that there are no fixed points of index $2k$, where $k$ is even and $0 \leq k \leq 2i$. Next, consider $\chi^{i+2}(M)=0$ and conclude that there are no fixed points of index $2k$ where $k$ even and $k \leq i+4$. And then we consider $\chi^{i+4}(M)=0$ to conclude that there are no fixed points of index $2k$ where $k$ is even and $k \leq i+6$. We continue this until $\chi^{n-i}(M)$ if $n$ is even and $\chi^{n-i-1}(M)$ if $n$ is odd, to conclude that there are no fixed points of index that is a multiple of 4, which contradicts Lemma \ref{c19} below that there must be fixed points whose indices differ by 2.

Third, suppose that $n$ is odd, $i>\frac{n}{2}$, and $i$ is odd. Considering $\chi^i(M)=0$, it follows that there are no fixed points of index $2k$ such that $k$ is even and $0 \leq k \leq 2(n-i)$. By Theorem \ref{t28}, since $N^j=N^{n-j}$ for all $j$, there are no fixed points of index $2k$, where $k$ is odd and $n-(2n-2i)=2i-n \leq k \leq n$. In particular, there are no fixed points of index $i-2$. Next, considering $\chi^{i-2}(M)=0$, we have that there are no fixed points of index $2k$, where $k$ is even and $k \leq 2n-2i+2$. By the symmetry that $N^j=N^{n-j}$ for all $j$, there are no fixed points of index $2k$ such that $k$ is odd and $2i-n-2 \leq k \leq n$. We continue this to have that there are no fixed points of any index, which is a contradiction.

As a slight variation of the arguments above, the other cases, (4) $n$ is odd, $i>\frac{n}{2}$, and $i$ is even, (5) $n$ is even, $i>\frac{n}{2}$, and $i$ is odd, and (6) $n$ is even, $i>\frac{n}{2}$, and $i$ is even, are proved. \end{proof}






\begin{cor} \label{c13} \cite{TW}, \cite{L}
A semi-free, symplectic circle action on a closed symplectic manifold $M$ with isolated fixed points is Hamiltonian if and only if it has a fixed point.
\end{cor}

\begin{proof}
All the weights are either $1$ or $-1$. Therefore, $A_i=\{i\}$ for each $i$ and so the corollary follows. \end{proof}

We show that, in certain cases, we can only look at the sum of three weights. The first instance is the following (Theorem \ref{t44}):


\begin{theo} \label{t14}
Consider a symplectic circle action on a $2n$-dimensional closed symplectic manifold $M$ with non-empty fixed points, whose weights are $\{\pm a_1, \pm a_2, \cdots, \pm a_n\}$ for some positive integers $a_i$, where $1 \leq i \leq n$. Assume that $n \leq 5$ and $\pm a_i \pm a_j \pm a_k \neq 0$ for all $i<j<k$. Then the action is Hamiltonian. 
\end{theo}

\begin{proof}
Assume, on the contrary, that the action in not Hamiltonian. By Theorem \ref{t28}, $\chi^0(M)=\chi^n(M)=0$ and there are no fixed points of index $0$ and $2n$. Denote by $A=\{a_1, a_2, \cdots , a_n\}$ and $A_i=\{a_{j_1}+a_{j_2}+\cdots+a_{j_i}\}_{a_{j_1}<a_{j_2}<\cdots<a_{j_i}}$ the collection of sums of $i$ elements of $A$, where $1 \leq i \leq n$. Then the problem is equivalent to showing that if $A_1 \cap A_2 =\emptyset$, the action is Hamiltonian. We consider $A_i \cap A_j$ for all $i,j$ such that $i \neq j \mod 2$, $1 \leq i \leq n-1$, $1 \leq j \leq n$.


First, assume that $n \leq 3$. Then $A_1 \cap A_2$ is the only intersection that we consider, so the result follows from Theorem \ref{t12}.

Second, assume that $n=4$. Then $A_1 \cap A_2$ and $A_2 \cap A_3$ are the only ones that we consider. However, $A_1 \cap A_2=\emptyset$ if and only if $A_2 \cap A_3 = \emptyset$. Therefore the result follows from Theorem \ref{t12}.

Finally, assume that $n=5$. The only possible non-empty intersections that we consider are $A_1 \cap A_4 \neq \emptyset$, $A_2 \cap A_3 \neq \emptyset$, and $A_3 \cap A_4 \neq \emptyset$. However, $A_2 \cap A_3 \neq \emptyset$ if and only if $A_3 \cap A_4 \neq \emptyset$. Therefore, we can only consider the case $A_1 \cap A_4 \neq \emptyset$ and the case $A_2 \cap A_3 \neq \emptyset$. By the assumption that $A_1 \cap A_2 = \emptyset$, if one is satisfied, the other fails to be satisfied. 

We consider $\chi^0(M)$. By Theorem \ref{t28},
\begin{center}
$\D{\chi^0(M) = \sum_{p \in M^{S^{1}}} \frac{1}{\prod_j (1-t^{\xi_{p}^{j}})}=\sum_{p \in M^{S^{1}}} (-1)^{\frac{\lambda_{p}}{2}} \frac{\prod_{\xi_{p}^{j}<0} t^{-\xi_{p}^{j}}           }{\prod_j (1-t^{\xi_{p}^{j}})}=0}.$
\end{center}
Multiplying by $(1-t^{a_1}) \cdots (1-t^{a_5})$ on the equation above, we have
\begin{center}
$\D{0=\sum_{p \in M^{S^{1}}} (-1)^{\frac{\lambda_{p}}{2}} \prod_{\xi_{p}^{j}<0} t^{-\xi_{p}^{j}} }.$ 
\end{center}

Assume first that $A_1 \cap A_4 \neq \emptyset$. Without loss of generality, let $a_5=a_1+a_2+a_3+a_4$. Suppose that there is a fixed point whose weights are $\{-a_1,a_2,a_3,a_4,a_5\}$. In the last expression, such a point has a summand $-t^{a_1}$. This term can only be cancelled by another term whose exponent is the sum of even elements. However, $a_5=a_1+a_2+a_3+a_4$ is the only equation among weights. Therefore, there cannot be a fixed point whose weights are $\{-a_1,a_2,a_3,a_4,a_5\}$. Similarly, one can show that we may only have fixed points whose weights are $\{a_1,a_2,a_3,a_4,-a_5\}$ and $\{-a_1,-a_2,-a_3,-a_4,a_5\}$. Therefore, we only have fixed points of index 2 and 8, which contradicts Corollary \ref{c19} below.

Next, suppose that $A_2 \cap A_3 \neq \emptyset$. Without loss of generality, let $a_1+a_2+a_3=a_4+a_5$. By using an argument similar to the case above, one can show that we may only have fixed points whose weights are $\{a_1,a_2,a_3,-a_4,-a_5\}$ and $\{-a_1,-a_2,-a_3,a_4,a_5\}$. Suppose that there are $k$ fixed points whose weights are $\{a_1,a_2,a_3,-a_4,-a_5\}$. By Theorem \ref{t28}, $k=\chi^2(M)=-\chi^3(M)$ and there are $k$ fixed points of weights $\{-a_1,-a_2,-a_3,a_4,a_5\}$. Moreover,
\begin{center}
$\D{0=\chi^1(M)= k \frac{t^{a_1}+t^{a_2}+t^{a_3}+t^{-a_4}+t^{-a_5}}{(1-t^{a_1})(1-t^{a_2})(1-t^{a_3})(1-t^{-a_4})(1-t^{-a_5})}}$

$\D{+ k \frac{t^{-a_1}+t^{-a_2}+t^{-a_3}+t^{a_4}+t^{a_5}}{(1-t^{-a_1})(1-t^{-a_2})(1-t^{-a_3})(1-t^{a_4})(1-t^{a_5})}}$

$\D{=k \frac{t^{a_4+a_5}(t^{a_1}+t^{a_2}+t^{a_3}+t^{-a_4}+t^{-a_5})}{(1-t^{a_1})(1-t^{a_2})(1-t^{a_3})(1-t^{a_4})(1-t^{a_5})}}$

$\D{-k \frac{t^{a_1+a_2+a_3}(t^{-a_1}+t^{-a_2}+t^{-a_3}+t^{a_4}+t^{a_5})}{(1-t^{a_1})(1-t^{a_2})(1-t^{a_3})(1-t^{a_4})(1-t^{a_5})}}$.
\end{center}
Multiplying the equation above by $(1-t^{a_1})(1-t^{a_2})\cdots(1-t^{a_5})$, we have
\begin{center}
$\D{0=k t^{a_4+a_5}(t^{a_1}+t^{a_2}+t^{a_3}+t^{-a_4}+t^{-a_5})}$ 
$\D{-k t^{a_1+a_2+a_3}(t^{-a_1}+t^{-a_2}+t^{-a_3}+t^{a_4}+t^{a_5})}$
$\D{=k(t^{a_1+a_4+a_5}+t^{a_2+a_4+a_5}+t^{a_3+a_4+a_5}+t^{a_5}+t^{a_4})}$
$\D{-k(t^{a_2+a_3}+t^{a_1+a_3}+t^{a_1+a_2}+t^{a_1+a_2+a_3+a_4}+t^{a_1+a_2+a_3+a_5})}$.
\end{center}
By the assumption it follows that the term $kt^{a_4}$ cannot be cancalled out, which is a contradiction. \end{proof}

Another case where the condition that the sum of any three weights is never equal to zero guarantees that a symplectic action is Hamiltonian, is when the dimension of the manifold is six (Theorem \ref{t43}). In fact, we prove a stronger result:

\begin{theo} \label{t15}
Consider a symplectic circle action on a six-dimensional closed symplectic manifold with non-empty isolated fixed points. Suppose that each negative weight at the fixed point of index 2 is never equal to the sum of negative weights at the fixed point of index 4. Then the action is Hamiltonian.

\end{theo}

\begin{proof}
Assume, on the contrary, that the action is not Hamiltonian. By Theorem \ref{t28}, $\chi^0(M)=\chi^3(M)=0$ and there are no fixed points of index $0$ and $6$. Moreover, the number of fixed points of index 2 and that of 4 are equal. Suppose that there are $k$ fixed points of index 2, and let $p_i,q_i$ be the fixed points of index 2,4, respectively, for $1 \leq i \leq k$. Let $\Sigma_{p_i}=\{-b_i,c_i,d_i\},\Sigma_{q_i}=\{-e_i,-f_i,g_i\}$ be the weights at $p_i,q_i$, respectively, where $b_i,c_i,d_i,e_i,f_i,g_i$ are positive integers. By permuting $p_i$'s and $q_i$'s if neccesary, we may assume that $b_1 \leq b_2 \leq \cdots \leq b_k$ and $e_1+f_1 \leq e_2+f_2 \leq \cdots \leq e_k+f_k$. By Theorem \ref{t28},
\begin{center}
$\D{\chi^0(M) = \sum_i \frac{1}{(1-t^{-b_i})(1-t^{c_i})(1-t^{d_i})} + \sum_i \frac{1}{(1-t^{-e_i})(1-t^{-f_i})(1-t^{g_i})}}$

$\D{=  - \sum_i \frac{t^{b_i}}{(1-t^{b_i})(1-t^{c_i})(1-t^{d_i})} + \sum_i \frac{t^{e_i+f_i}}{(1-t^{e_i})(1-t^{f_i})(1-t^{g_i})}=0.}$
\end{center}

Denote by $A=\{a_1,a_2, \cdots, a_l\}$ the collection of all the absolute values of the weights over all the fixed points counted with multiplicity. Let $B_i=A \setminus \{b_i,c_i,d_i\}$,$C_i=A \setminus \{e_i,f_i,g_i\}$ be the elements in $A$ minus the absolute values of weights at $p_i$,$q_i$, respectively.

Multiplying the equation above by $\prod_{a \in A} (1-t^{a})$, we have
\begin{center}
$0= - \sum_i t^{b_i} \prod_{a \in B_i} (1-t^{a}) + \sum_i t^{e_i+f_i} \prod_{a \in C_i} (1-t^{a}).$
\end{center}

In this equation, a term has the coefficient -1 if its exponent is the sum of odd elements and 1 if its exponent is the sum of even elements. Consider $-t^{b_1}$. Since $b_1 \leq b_i$ for $i \geq 2$, this term cannot be cancelled out by any summand in $-\sum_i t^{b_i} \prod_{a \in B_i} (1-t^{a})$. Therefore, it must be cancelled out by another summand in $\sum_i t^{e_i+f_i} \prod_{a \in C_i} (1-t^{a})$ for some $i$, whose exponent is the sum of even elements, where at least two elements of them are $e_i,f_i$. By the assumption, the exponent of such a summand cannot be $e_i+f_i$. Hence, the exponent of the term must be the sum of at least four elements, say $b_1=e_j+f_j+ \cdots$. Next, consider $t^{e_1+f_1}$. We have that $e_1+f_1 < e_j+f_j+\cdots \leq b_1$. Since $e_1+f_1 \leq e_2+f_2 \leq \cdots \leq e_k+f_k$, the term $t^{t^{e_1+f_1}}$ cannot be cancelled out by any term in $\sum_i t^{e_i+f_i} \prod_{a \in C_i} (1-t^{a})$. On the other hand, $e_1+f_1 < b_1$. Therefore, it cannot also be cancelled out by any term in $-\sum_i t^{b_i} \prod_{a \in B_i} (1-t^{a})$, which is a contradiction. \end{proof}


As a corollary, we recover the result by L. Godinho:

\begin{cor} \label{c16} \cite{G} Let the circle act symplectically on a six-dimensional closed symplectic manifold. Suppose that fixed points are isolated and their weights are $\{ \pm a, \pm b, \pm c\}$, where $0<a \leq b \leq c$ and $a+b \neq c$. If there is a fixed point, then the action is Hamiltonian. \end{cor}

\begin{proof} This follows from Theorem \ref{t12}, Theorem \ref{t14}, or Theorem \ref{t15}. \end{proof}

For a certain type of weights, we show that there is a restriction. To show the restriction, we introduce a terminology.

\begin{Definition}
Consider a circle action on a closed almost complex manifold. Suppose that the action preserves the almost complex structures and the fixed points are isolated. Denote by $A=\{a_1, a_2, \cdots , a_l\}$ the collection of all the absolute values of weights among all the fixed points counted with multiplicity, and $A_i=\{a_{j_1}+a_{j_2}+\cdots+a_{j_i}\}_{a_{j_1}<a_{j_2}<\cdots<a_{j_i}}$ the collection of sums of $i$ elements of $A$, for $1 \leq i \leq l$. A positive weight $w$ is called \textbf{primitive}, if $w \notin A_i$ for $i \geq 2$, i.e. $w$ is never equal to the sum of the absolute values of weights among all the fixed points, counted with multiplicity, other than $w$ itself.
\end{Definition}

Note that the smallest positive weight is primitive. In \cite{K}, C. Kosniowski derives a certain formula for a holomorphic vector field on a complex manifold with only simple isolated zeros. We follow the idea of C. Kosniowski to find a restriction for a primitive weight of a circle action on a closed almost complex manifold with isolated fixed points. For the smallest positive weight, the Lemma is already given in \cite{JT} and the proof is almost identical, but we give a proof in details.

\begin{lem} \label{l17}
Consider a circle action on a $2n$-dimensional closed almost complex manifold. Suppose that the action preserves the almost complex structure and the fixed points are isolated. For each primitive weight $w$, the number of times the weight $-w$ occurs at fixed points of index $2i$ is equal to the number of times the weight $w$ occurs at fixed points of index $2i-2$, for all $i$.
\end{lem}

\begin{proof}
We first show that 
\begin{center}
$\D{\sum_{\lambda_{p}=2i} [N_{p}(w)+N_{p}(w)]=\sum_{\lambda_{p}=2i-2} N_{p}(w)+\sum_{\lambda_{p}=2i+2} N_{p}(-w)}, (\ast)$
\end{center}
where $N_p(w)$ is the number of times the weight $w$ occurs at $p$. By Theorem \ref{t28},
\begin{center}
$\D{\chi^i(M) = \sum_{p \in M^{S^1}} \frac{\sigma_i (t^{\xi_p^1}, \cdots, t^{\xi_p^n})}{\prod_{m=1}^{n} (1-t^{\xi_p^m})}}$

$\D{=\sum_{p \in M^{S^1}} (-1)^{\frac{\lambda_{p}}{2}} \frac {[\prod_{\xi_p^m<0} t^{-\xi_p^m}] \sigma_i (t^{\xi_p^1}, \cdots, t^{\xi_p^n})}{\prod_{m=1}^{n} (1-t^{|\xi_p^m|})}}.  (\ast\ast)$ 
\end{center}
Denote by $\D{J_p=[\prod_{\xi_p^m<0} t^{-\xi_p^m}] \sigma_i (t^{\xi_p^1}, \cdots, t^{\xi_p^n})}$ and $\D{K_p=\prod_{m=1}^{n} (1-t^{|\xi_p^m|})}$.

If $\lambda_{p}=2i$, then $J_p=1+f_p(t)$, where $f_p(t)$ is a polynomial that does not have a constant term and $t^w$-term.

If $\lambda_{p}=2i \pm 2$, then $J_p=N_{p}(\mp w)t^w+f_p(t)$, where $f_p(t)$ is a polynomial that does not have a constant term and $t^w$-term.

If $\lambda_{p} \neq 2i, 2i \pm2$, then $J_p=f_p(t)$, where $f_p(t)$ is a polynomial that does not have a constant term and $t^w$-term.

Multiplying $(\ast \ast)$ by $\displaystyle{\prod_{p \in M^{S^1}} K_p}$ yields
\begin{center}
$\chi^i(M)[1- \sum_p(N_p(-w)+N_p(w))t^w]+g_1(t)= \{(-1)^{i-1} \sum_{\lambda_{p}=2i-2} N_{p}(w) + (-1)^{i+1} \sum_{\lambda_{p}=2i+2} N_{p}(-w) +  (-1)^i \sum_{\lambda_{p}=2i} ( N_{p}(w)+N_{p}(-w))- \chi^i(M)\sum_{p} ( N_{p}(w)+N_{p}(-w))\}t^w + \sum_{\lambda_{p}=2i}(-1)^i +g_2(t)$,
\end{center}
where $g_i(t)$ are polynomials without constant terms and $t^w$-terms. Comparing the coefficients of $t^w$-terms, the claim follows.

Applying $(\ast)$ for $i=0$, we have
\begin{center}
$\D{\sum_{\lambda_{p}=0} N_{p}(w)=\sum_{\lambda_{p}=2} N_{p}(-w)}$.
\end{center}

Next, applying $(\ast)$ for $i=1$, we have
\begin{center}
$\D{\sum_{\lambda_{p}=2} [N_{p}(-w)+N_{p}(w)]=\sum_{\lambda_{p}=0} N_{p}(w)+\sum_{\lambda_{p}=4} N_{p}(-w)}$.
\end{center}

Since $\D{\sum_{\lambda_{p}=0} N_{p}(w)=\sum_{\lambda_{p}=2} N_{p}(-w)}$, it follows that 
\begin{center}
$\D{\sum_{\lambda_{p}=2} N_{p}(w)=\sum_{\lambda_{p}=4} N_{p}(-w)}$,
\end{center}

Continuing this, the Lemma follows.
\end{proof}

As an application, there must be at least two fixed points whose indices are nearby. This is shown for a holomorphic vector field on a compact complex manifold with only simple isolated zeroes by C. Kosniowski \cite{K}.

\begin{cor} \label{c19}
Consider a circle action on a closed almost complex manifold. Suppose that the action preserves the almost complex structure and the fixed points are non-empty and isolated. Then there exist two fixed points whose indices differ by 2.
\end{cor}

\begin{proof}
Apply Lemma \ref{l17} to the smallest positive weight.
\end{proof}


\end{document}